\documentclass[12pt,reqno]{amsart}
\usepackage[a4paper,margin=1in]{geometry}

\usepackage{microtype}
\usepackage[cal=euler,scr=boondoxo]{mathalfa}
\usepackage[colorlinks,allcolors=blue]{hyperref}
\usepackage[capitalize]{cleveref}
\crefname{equation}{}{}

\hypersetup{
  pdftitle={Quasi-modularity of symmetric quasi-shuffles},
  pdfauthor={Killian Hong-Minh and Sergey Mozgovoy}
}

\usepackage{etoolbox,xspace,bbm}

\def\newthm#1{
  \newtheorem{#1}[thm]{\MakeUppercase#1}
  \AddToHook{env/#1/begin}{\crefalias{thm}{#1}} 
}
\theoremstyle{plain}
\forcsvlist\newthm{theorem,lemma,corollary,proposition,conjecture}

\newtheorem{innercustomthm}{}
\newenvironment{customthm}[1]
  {\def\theinnercustomthm{#1}\innercustomthm}
  {\endinnercustomthm}

\theoremstyle{definition}
\forcsvlist\newthm{definition,remark,example,question,claim}

\def\opn{\operatorname}
\def\opr#1#2{\def#1{\opn{#2}}}
\def\csopr#1#2{\csdef{#1}{\opn{#2}}}
\def\oper#1{\csopr{#1}{#1}}
\def\operl{\forcsvlist\oper}

\forcsvlist\oper{fs,sf,Bl,GL,SL,deg,Hom,End,Vect,Mod,Rep,Ker,Im,Id,trop,tdet,vol, lt, rk, sgn, FY,id,d,Res,Ind,Gr,Spec,Coker,Ob,Sh,Ab,supp,Ext,Exp,ch,Aut,tr,ad,td,Stab,QSym}
\def\Ker{\operatorname{Ker}}

\def\defbb#1{\csdef{b#1}{\mathbb{#1}}}
\forcsvlist\defbb{A,B,C,D,E,F,G,H,I,J,K,L,M,N,O,P,Q,R,S,T,U,V,W,X,Y,Z}

\def\defcal#1{\csdef{c#1}{\mathcal{#1}}}
\forcsvlist\defcal{A,B,C,D,E,F,G,H,I,J,K,L,M,N,O,P,Q,R,S,T,U,V,W,X,Y,Z}

\def\deffrak#1{\csdef{k#1}{\mathfrak{#1}}}
\forcsvlist\deffrak{A,B,C,D,E,F,G,H,I,J,K,L,M,N,O,P,Q,R,S,T,U,V,W,X,Y,Z,a,b,c,d,e,f,g,h,i,j,k,l,m,n,o,p,q,r,s,t,u,v,w,x,y,z}

\def\defsf#1{\csdef{s#1}{\mathsf{#1}}}
\forcsvlist\defsf{A,B,C,D,E,F,G,H,I,J,K,L,M,N,O,P,Q,R,S,T,U,V,W,X,Y,Z,a,b,c,d,e,f,g,h,i,j,k,l,m,n,o,p,q,r,s,t,u,v,w,x,y,z,QS}

\def\al{\alpha}

\def\ga{\gamma}
\def\de{\delta}
\def\la{\lambda}
\def\ta{\tau}
\def\eps{\varepsilon}

\def\si{\sigma}

\def\Ga{\Gamma}

\def\La{\Lambda}

\def\set#1{\left\{#1\right\}}
\def\sets#1#2{\left\{#1\ \middle\vert\ #2\right\}}
\def\rbr#1{\left(#1\right)}
\def\sbr#1{\left[#1\right]}

\def\n#1{\left\lvert#1\right\rvert}

\def\ov#1#2{{\substack{#1\\#2}}} 

\def\smat#1{\rbr{\begin{smallmatrix}#1\end{smallmatrix}}}

\def\ol#1{\overline{#1}}
\def\pser#1{[\![#1]\!]} 
\def\lser#1{(\!(#1)\!)} 

\def\mto{\mapsto}

\def\ts{\otimes}
\def\xx{\times}
\def\iso{\simeq}
\def\sbs{\subset}
\def\tpdf{\texorpdfstring}

\RequirePackage[framemethod=TikZ]{mdframed}
\mdfsetup{backgroundcolor=brown!10, linewidth=0,
  skipabove=5pt, roundcorner=5pt}

\theoremstyle{definition}
\newmdtheoremenv[innertopmargin=-2pt]{note}{Note}

\DeclareDocumentCommand{\idef}{mo}{%
  \ifstrempty{#1}{}{\ifmmode#1\else\emph{#1}\fi}%
	\IfValueTF{#2}{\index{#2}}{\index{#1}}%
}
\def\eg{e.g.\ } 
\def\cf{cf.\ } 
\def\eq#1{\begin{equation}#1\end{equation}}

\usepackage[T2A]{fontenc} 
\def\sha{\mathbin{\textup\cyrsh}}
\def\rsbox#1#2#3{
  \raisebox{#1}{\scalebox{#2}{\ensuremath{#3}}}}
\def\isoto{\xto{\,\smash{\rsbox{-3pt}{.8}\sim}\,}}
\NewDocumentCommand\ubar{O{1.5}O{1.5}m}{%
  \mkern#1mu\overline{\mkern-#1mu#3\mkern-#2mu}\mkern#2mu}

\def\prt{\vdash} 
\def\bop{\bigoplus}
\def\i{^{-1}}
\def\xto{\xrightarrow}

\operl{sh,qsh,Hopf,Alg}
\operl{QS,TS,SQS}
\oper{ord}
\opr\M{\mathbf M} 
\opr\QM{\mathbf{QM}} 
\opr\SS{\mathsf{S}} 

\def\bff{\mathbf f}
\def\con{\mathrm{con}}
\def\nun{\mathrm{nu}} 
\def\bdd{\mathrm{bd}}

\def\bn{\mathbf n}

\def\pN{\bZ_{\ge1}}
\def\bbR{\ubar[3][0] R}
\def\QQ{\bQ\pser q}
\def\QQp{\QQ^+}

\mathchardef\col\mathcode`:
\mathchardef\latexequal\mathcode`=
\AtBeginDocument{\mathcode`:="8000\mathcode`=="8000}
\RequirePackage{mathtools} 

\def\colona{\nobreak\mskip2mu\mathpunct{}
  \mkern-\thinmuskip{\col}\nonscript\mskip6mu plus1mu}

{\makeatletter
  \catcode`:=\active\catcode`==\active
  \gdef:{\@ifnextchar={\coloneqq\@gobble}{\colona}}
  \gdef={\@ifnextchar:{\eqqcolon\@gobble}{\latexequal}}
}

\begin{document}
\title{Quasi-modularity of symmetric quasi-shuffles}

\author{Killian Hong-Minh}
\author{Sergey Mozgovoy}

\address{School of Mathematics, Trinity College Dublin, Dublin 2, Ireland
\newline\indent
Hamilton Mathematics Institute, Dublin 2, Ireland}

\email{hongmink@tcd.ie}
\email{mozgovoy@maths.tcd.ie}

\begin{abstract}
We develop an algebraic framework for the quasi-modularity of symmetric multiple $q$-zeta values. We identify natural classes of symmetric quasi-shuffles whose $q$-zeta values exhaust the algebra of level-one quasi-modular forms, and further classes whose $q$-zeta values are quasi-modular forms of finite level. We also obtain explicit symmetrization formulas and relate symmetric quasi-shuffles to symmetric and quasisymmetric functions.
\end{abstract}

\maketitle

\section{Introduction}
Given $f_1,\dots,f_k\in B^+=q\QQ$,
consider the multiple $q$-zeta value
\eq{Z(f_1\ts\dots\ts f_k)
=\sum_{0<n_1<\dots<n_k}\prod_{i=1}^k f_i(q^{n_i})\in \QQ.}
This extends to an algebra morphism
\eq{Z:\QS(B^+)=\bop_{k\ge0}(B^+)^{\ts k}\to \QQ,
}
where $\QS(B^+)$ is the quasi-shuffle algebra of the non-unital algebra $B^+$.
Quasi-shuffle algebras of non-unital associative algebras were introduced in \cite{newman_cofree} under the name
\emph{cofree irreducible Hopf algebras}.
See also \cite{hoffman_quasia,fares_quelques}.

Traditionally, when studying $Z$-values,
one considers an algebra $A\sbs B=\QQ$
and fixes a basis
$\bff=(f_i)_{i\ge1}$ of $A^+=A\cap B^+$
(for example, $f_i=\frac{q^i}{(1-q)^i}$ for
$A=\bQ\sbr{q/\rbr{1-q}}$).
Let
\eq{
Z^\bff(\al_1,\dots,\al_k)
=Z(f_{\al_1}\ts\dots\ts f_{\al_k}),
\qquad \al\in\pN^k.
}
These values span the subalgebra
$Z(\QS(A^+))\sbs\QQ$.
This basis-dependent approach provides a convenient notation for the resulting values, although it may obscure the underlying algebraic structure.
We will be interested in the symmetrizations
\eq{Z^\bff_\la
=\sum_{\al^+=\la}Z^\bff(\al)
=\sum_{\al^+=\la}Z(f_{\al_1}\ts\dots\ts f_{\al_k})\in \QQ,}
where $\la$ is a partition of length $k$,
and $\al^+$, for $\al\in\pN^k$,
denotes the partition obtained by rearranging the entries of $\al$.
For example, $Z^\bff_{(r^k)}=Z(f_r^{\ts k})$.

Many interesting $q$-series arise as $Z$-values.
For example, consider the Eisenstein series (for even $k>0$)
\eq{E_k=1-\frac{2k}{B_k}\sum_{n,d\ge1}n^{k-1}q^{nd},}
where $B_k\in\bQ$ are Bernoulli numbers satisfying $\sum_{k\ge0}B_kt^k/k!=\frac t{e^t-1}$.
Using the operator $D=q\frac d{dq}$ and the Adams operations $\psi^d(f)=f(q^d)$, we can write
\eq{
\sum_{n,d\ge1}n^{k-1}q^{nd}
=\sum_{d\ge1}\psi^d\rbr{D^{k-1}\frac q{1-q}}
=Z(Q_k),\qquad Q_k=D^{k-1}\frac q{1-q}.
}
Therefore $Z(Q_k)$ is equal to $E_k-1$ up to a rational scalar.
The Eisenstein series $E_{2k}$ for $k\ge2$ belong to the algebra of modular forms $\M(1;\bQ)=\bQ[E_4,E_6]$.
This algebra is not closed under the derivation $D$,
whereas the algebra of quasi-modular forms $\QM(1;\bQ)=\bQ[E_2,E_4,E_6]$ is.

Another example is the MacMahon series
\cite{macmahon_divisors}
\eq{\label{MM}
\sum_{0<n_1<\dots<n_k}\prod_{i=1}^k\frac{q^{n_i}}{(1-q^{n_i})^2}
=Z(Q_2^{\ts k}),
}
where $Q_2=D\frac q{1-q}=\frac q{(1-q)^2}$.
It was proved in \cite{andrews_macmahons,rose_quasi} that this series is quasi-modular.
Quasi-modularity of similar series was also studied in
\cite{amdeberhan_macmahons,nazaroglu_quasimodularity,
bringmann_limiting,kang_quasi,
bachmann_macmahons,
shintani_generalization}.
This paper arose from our attempt to give a purely algebraic
explanation,
rather than a number-theoretic or combinatorial one,
of the above phenomenon.
As a result, we develop an algebraic framework for proving quasi-modularity of a large class of $Z$-values,
including~\cref{MM}.
\medskip

Let $Q_1=\frac{q}{1-q}$ and $Q_2=DQ_1=\frac{q}{(1-q)^2}$ be defined as before.
Consider the algebras
\[R_2=\bQ[Q_2]\sbs R_1=\bQ[Q_1]\sbs R=\bQ\sbr{q,\tfrac1{q-1}}.\]
The image $\cZ_{q}$ of the algebra morphism
\[Z:\QS(R^+_1)\to\QQ\]
is called the \idef{algebra of multiple $q$-zeta values}
(see \eg \cite{bachmann_dimension,hirose_unified,mozgovoy_multiple}).
It contains the algebra $\QM(1;\bQ)$ of quasi-modular forms.
We have
\[R_1=\bop_{k\ge0}\bQ Q_{k},\qquad
R_2=\bop_{k\ge0}\bQ Q_{2k},\]
where $Q_0=1$.
An alternative description of $R_2$ is
\[R_2
=R_1^\ta=R^\ta,\]
where
$\ta$ is an involution of $\bQ(q)$ given by $\ta(f)=f(q\i)$,
and $R^\ta$ is the $\ta$-invariant locus of $R$.
\medskip

In this paper, we study the algebra morphism
$Z:\QS(R^+_2)\to\QQ$ restricted to symmetric tensors.
More precisely, given a (non-unital) algebra $A$, define the subspace of symmetric tensors
\[\TS(A)=\bop_{n\ge0}(A^{\ts n})^{S_n}\sbs
\bop_{n\ge0}A^{\ts n}=\QS(A).\]
We will prove in \cref{SQS:th} that it is a subalgebra of $\QS(A)$
generated by $A$.

\begin{theorem}\label{thm1}
The image of the map
\[Z:\TS(R^+_2)\to\QQ\]
is equal to $\QM(1;\bQ)$, the algebra of (rational, level 1) quasi-modular forms.
\end{theorem}

In particular, the series $Z(Q_{2r}^{\ts k})$ is quasi-modular for all $k,r\ge1$.
More generally, for every sequence
$\bff=(f_i)_{i\ge1}$ in $R^+_2$
and for every partition $\la$,
the symmetrization $Z^\bff_\la$ is quasi-modular.
For example,
\[Z(Q_2\ts Q_4)+Z(Q_4\ts Q_2)\]
is quasi-modular.

If we also allow quasi-modular forms of finite level (as opposed to level one considered above),
then we can prove the following.
Consider the $\la$-ring closures
\[
\bbR_2=\bQ\sbr{\psi^n(Q_2)\col n\ge1}
\sbs
\bbR=\bQ[q,(q^n-1)\i\col n\ge1]\sbs\QQ.\]
We will prove in \cref{R2 closure} that $(\bbR)^\ta=\bbR_2$.

\begin{theorem}\label{thm2}
The image of the map
\[Z:\TS(\bbR_2^+)\to\QQ\]
is contained in $\QM(\infty;\bQ)$,
the algebra of (rational) finite level quasi-modular forms.
\end{theorem}


An interesting question is how far the image of the above map is
from $\QM(\infty;\bQ)$.
To prove \cref{thm2}, we will show that $Z(\bbR_2^+)$ consists of
quasi-modular forms.
This is done by applying a recent result of
Shintani \cite{shintani_generalization}.

Intuitively, the relation between \cref{thm1} and \cref{thm2}
is the following.
It is known that if $g\in\QM(N)$, then $\psi^d(g)=g(q^d)\in\QM(dN)$.
Therefore, if $Z(f)\in\QM(1)$ for some $f\in B^+$,
then $Z(\psi^d(f))\in\QM(d)$.
In particular, $Z(\psi^d(f))\in\QM(\infty)$ for all $f\in R^+_2$ by \cref{thm1}.
The elements $\psi^d(f)$, for $f\in R_2=R^\ta$, generate the algebra $(\bbR)^\ta$ (see \cref{R2 closure}),
so one may expect that $Z(f)\in\QM(\infty)$ for all $f\in\bbR^\ta\cap B^+$.
This is the statement of \cref{thm2}.


The paper is organized as follows.
In \cref{sec:qs}, we develop the algebraic theory of symmetric
quasi-shuffles and derive explicit symmetrization formulas.
In \cref{sec:qm}, we study the model algebras relevant to multiple $q$-zeta values and apply the algebraic results to quasi-modular forms of level one and of finite level.
\section{Quasi-shuffle algebras}
\label{sec:qs}

This section develops the algebraic ingredients of the paper.
After recalling quasi-shuffle and shuffle algebras, we study the
symmetric tensors inside a quasi-shuffle algebra.
The main result is that they form precisely the subalgebra generated
by the underlying algebra. We then give explicit symmetrization
formulas and relate the construction to symmetric and quasisymmetric
functions.

\subsection{Quasi-shuffle algebras}
We begin by recalling the quasi-shuffle construction and the
universal property that will be used later.
For more details on quasi-shuffle algebras
see~\eg~\cite{mozgovoy_multiple}.
Let $(A,\circ)$ be a (non-unital) algebra over a field $K$.
The \idef{quasi-shuffle algebra}
\[\QS(A)=\QS(A,\circ)=\bop_{n\ge0}A^{\ts n}\]
has the \idef{quasi-shuffle product} $*$ defined inductively by
$1_K*u=u*1_K=u$ and
\eq{
(au)*(bv)=a(u*bv)+b(au*v)+(a\circ b)(u*v),\qquad
a,b\in A,\, u,v\in\QS(A),
}
where $uv=u\ts v$ for $u,v\in\QS(A)$ is the concatenation product.
For example,
\[a*b=ab+ba+a\circ b,\qquad
a*bc=abc+bac+bca+(a\circ b)c+b(a\circ c),\qquad
a,b,c\in A.\]
Equivalently, this product can be described as follows.
Define the set of \idef{quasi-shuffles}
\eq{\qsh(m,n)=\bigsqcup_{r\ge0}\qsh(m,n;r),}
where $\qsh(m,n;r)$ consists of surjective maps $\si:[m+n]\to[r]$
satisfying
\[\si(1)<\dots<\si(m),\qquad
\si(m+1)<\dots<\si(m+n).\]
The elements of $\sh(m,n)=\qsh(m,n;m+n)$ are called \idef{shuffles}
and can be identified with the elements of $S_{m+n}/(S_m\xx S_n)$.
For $u=a_1\ts\dots\ts a_m$ and $v=b_1\ts\dots\ts b_n$
we have
\eq{u*v=\sum_{\si\in\qsh(m,n)}\si(u\ts v),}
where $\si(u\ts v)$ for $\si\in\qsh(m,n;r)$
is given by $c_1\ts\dots\ts c_r$ with
\[c_k=\begin{cases}
a_i&k=\si(i)\notin\si(m+[n]),\\
b_j&k=\si(m+j)\notin\si([m]),\\
a_i\circ b_j& k=\si(i)=\si(m+j).
\end{cases}\]

The algebra $\QS(A)$ is commutative if $A$ is commutative.
The algebra $\QS(A)$ is a connected (or conilpotent) Hopf algebra
with the coproduct given by deconcatenation
\[\de(a_1\ts\dots\ts a_n)
=\sum_{i=0}^n (a_1\ts\dots\ts a_i)\ts(a_{i+1}\ts\dots\ts a_n).
\]
Let $\Hopf_\con$ be the category of connected Hopf algebras and let $\Alg_\nun$ be the category of non-unital associative algebras over the field $K$.

\begin{theorem}[See {\cite{newman_cofree}}]
\label{adj}
The functor $\QS:\Alg_\nun\to\Hopf_\con$ is right adjoint to the augmentation-ideal functor
\[\Hopf_\con\to\Alg_\nun,\qquad
(H,\mu,\eta,\de,\eps)\mto\bar H=\Ker(\eps),
\]
where $\eps:H\to K$ is the counit of $H$.
\end{theorem}

\subsection{Shuffle algebras}
The shuffle algebra is the special case of the quasi-shuffle
construction in which the multiplication on the underlying vector
space is zero.
We recall it because symmetric tensors are particularly transparent in this case and provide the model for the construction in the next subsection.
For a vector space $V$,
define the \idef{shuffle algebra}
\[\Sh(V)=\QS(V,0),\]
where~$V$ is equipped with the zero product.
The product in $\Sh(V)$, called the \idef{shuffle product} and denoted by $\sha$, is given by
\eq{u\sha v
=\sum_{\si\in\sh(m,n)}\si(u\ts v),\qquad
u\in V^{\ts m},\, v\in V^{\ts n}.
}
Define the \idef{symmetric shuffle algebra}
or the \idef{algebra of symmetric tensors}
(\cf \cite[IV.5.3]{bourbaki_algebra4})
\[\TS(V)=\bop_{n\ge0}\TS^n(V)
=\bop_{n\ge0}(V^{\ts n})^{S_n}.
\]
It is a Hopf subalgebra of the shuffle algebra $\Sh(V)$.
The shuffle product in $\TS(V)$ can be written in the form
\eq{u\sha v=\sum_{[\si]\in S_{m+n}/(S_m\xx S_n)}\si(u\ts v),\qquad
u\in \TS^m(V),\, v\in \TS^n(V).
}
Define the \idef{symmetrization map}
\eq{\SS:V^{\ts n}\to \TS^n(V),\qquad
a_1\ts \dots \ts a_n\mto \sum_{\si\in S_n}\si(a_1\ts \dots \ts a_n).
}
By abuse of notation, for $a=(a_1,\dots,a_n)\in V^n$,
we write $\SS(a)=\SS(a_1,\dots,a_n)=\SS(a_1\ts \dots \ts a_n)$
and call it the \idef{symmetrization} of $a$.
We have
\begin{equation}\label{sha1}
a_1\sha\dots\sha a_n=\SS(a_1,\dots,a_n).
\end{equation}

\subsection{Symmetric quasi-shuffle algebras}
We now return to an arbitrary algebra $A$ and consider the symmetric
tensors $\TS(A)\sbs\QS(A)$.
Our main algebraic observation is that, despite the additional
contraction terms in the quasi-shuffle product, this space is still
a subalgebra and is generated by the degree-one part $A$.
This will reduce the quasi-modularity of symmetric $Z$-values to the
quasi-modularity of $Z$ on $A$.

Let $A$ be a (non-unital) algebra over a field $K$ of characteristic zero.
For $a\in A^n$ and a nonempty set $B=\set{i_1<\dots<i_k}\sbs[n]$,
let
\eq{\label{aB}
a_B=a_{i_1}\circ\dots\circ a_{i_k}\in A.}
Let $\Pi_n$ be the set of all set partitions of $[n]=\set{1,\dots,n}$ (identified with equivalence relations on $[n]$).
For $\pi=\set{B_1,\dots,B_k}\in\Pi_n$, let
\eq{\SS_\pi(a)=\SS(a_{B_1},\dots,a_{B_k})\in\TS^k(A).}

\begin{lemma}\label{qs-sh}
For $a\in A^n$, we have
\begin{equation}\label{qs-sh:eq}
a_1*\ldots*a_n
=\sum_{\pi\in\Pi_n}\SS_\pi(a).
\end{equation}
\end{lemma}
\begin{proof}
This follows directly from the formula for the quasi-shuffle product.
\end{proof}

\begin{theorem}\label{SQS:th}
The subalgebra $\SQS(A)\sbs\QS(A)$ generated by $A\sbs\QS(A)$
is equal to
\[\TS(A)=\bop_{n\ge0}(A^{\ts n})^{S_n}\sbs \QS(A).\]
It is a Hopf subalgebra of $\QS(A)$.
\end{theorem}
\begin{proof}
We have $\SQS(A)\sbs\TS(A)$ by \cref{qs-sh:eq} and the fact that
$\SS(a_1,\dots,a_n)$ for $a\in A^n$,
span $\TS(A)=\bop_{n\ge0}(A^{\ts n})^{S_n}$.

Let us show that $\SS(a_1,\dots,a_n)\in\SQS(A)$ by induction on $n$.
If $\pi\in\Pi_n$ has size $<n$, then $\SS_\pi(a)\in\SQS(A)$ by induction.
Otherwise, $\pi=\set{\set{1},\dots,\set n}$
and $\SS_\pi(a)=\SS(a_1,\dots,a_n)$.
Therefore  $\SS(a_1,\dots,a_n)\in\SQS(A)$
by \cref{qs-sh:eq}.
We conclude that $\SQS(A)=\TS(A)$.

The statement that $\SQS(A)$ is a Hopf subalgebra follows from the fact that $\de(a)=a\ts1+1\ts a\in \SQS(A)\ts \SQS(A)$ for $a\in A$.
\end{proof}

\begin{corollary}\label{Z of TS}
Let $Z:\QS(A)\to B$ be an algebra morphism such that $Z(A)\sbs C$, where $C\sbs B$ is a subalgebra.
Then $Z(\TS(A))\sbs C$.
\end{corollary}
\begin{proof}
Since $Z(A)\sbs C$ and $A$ generates $\SQS(A)=\TS(A)$,
we have $Z(\TS(A))\sbs C$.
\end{proof}

\begin{remark}
For a bialgebra $H$, its subspace of primitive elements
is defined by
\[P(H)=\sets{x\in H}{\de(x)=x\ts1+1\ts x}.\]
It is a Lie subalgebra of $H$,
and there is an algebra morphism $U(P(H))\to H$,
which is injective if the base field has characteristic zero
(we always assume this).
For $H=\QS(A)$ we have $P(H)=A$,
hence there is an algebra isomorphism
$U(A)\isoto\SQS(A)\sbs\QS(A)$,
where $A$ is regarded as a Lie algebra via the commutator bracket.
If $A$ is commutative,
this gives an algebra isomorphism
$S(A)\isoto\SQS(A)$.
This is not too surprising, since the quasi-shuffle algebra $\QS(A)$ is isomorphic to the shuffle algebra $\Sh(A)$ by
\cite{newman_cofree,hoffman_quasia}.
\end{remark}

\subsection{Symmetrization formulas}
The preceding theorem shows abstractly that every symmetrizer can be
expressed in terms of quasi-shuffle products.
We now make this expression explicit.
For an arbitrary associative algebra the formula is naturally indexed
by permutations; in the commutative case it reduces to a formula
indexed by set partitions, whose coefficients are the corresponding
M\"obius coefficients of the poset $\Pi_n$.

Consider the sign character
\[\sgn: S_n\to\set{\pm1},\qquad \si\mto(-1)^{n-c(\si)},\]
where $c(\si)$ is the number of cycles of $\si\in S_n$.
Every cycle of $\si$ can be written in the form $C=(i_1\dots i_r)$, where $i_1$ is the minimal element of $C$.
For $a\in A^n$, let
\eq{a_C=a_{i_1}\circ\dots \circ a_{i_r}\in A.}
We order the cycles $C_1,\dots,C_{c(\si)}$ of $\si$ by their minimal elements and define
\eq{P_\si(a)=a_{C_1}*\dots*a_{C_{c(\si)}}\in\SQS(A).}

\begin{theorem}[Symmetrization formula]
\label{sym formula1}
For $a\in A^n$, we have
\[
\SS(a_1,\dots,a_n)
=\sum_{\si\in S_n}\sgn(\si)P_\si(a).
\]
\end{theorem}
\begin{proof}
Let $F_n(a)$ be the right-hand side of the equation.
The quasi-shuffle product with the one-letter word $a_n$ gives
\eq{\label{S-ind}
\SS(a_1,\dots,a_{n-1})*a_n
=\SS(a_1,\dots,a_n)
+\sum_{i=1}^{n-1}
\SS(a_1,\dots,a_{i-1},a_i\circ a_n,a_{i+1},\dots,a_{n-1}).
}
We will show that $F_n$ satisfies the same recursion. Split the permutations
$\sigma\in S_n$ according to whether $n$ is a singleton cycle.
In the first case, deleting the cycle $(n)$ gives $\tau\in S_{n-1}$,
and the corresponding contribution to $F_n(a)$ is
$F_{n-1}(a_1,\dots,a_{n-1})*a_n$.

In the second case, let $i$ be the element immediately preceding $n$
in its cycle.
Deleting $n$ from the cycle gives $\tau\in S_{n-1}$
with $c(\si)=c(\ta)$ and $\sgn\si=-\sgn \ta$
such that
\[
P_\si(a)=P_\tau(a_1,\dots,a_{i-1},a_i\circ a_n,a_{i+1},\dots,a_{n-1}).
\]
This defines a bijection between such permutations $\si$ and pairs
$(\tau,i)\in S_{n-1}\times[n-1]$.
Hence
\[
F_n(a)
=F_{n-1}(a_1,\dots,a_{n-1})*a_n-\sum_{i=1}^{n-1}
F_{n-1}(a_1,\dots,a_{i-1},a_i\circ a_n,a_{i+1},\dots,a_{n-1}).
\]
Comparison with \cref{S-ind},
together with $F_1(a_1)=a_1$, proves the
claim by induction.
\end{proof}

\begin{example}
We have
\begin{gather*}
\SS(a,b)=a*b-a\circ b,\\
\SS(a,b,c)
=a*b*c-(a\circ b)*c-(a\circ c)*b
-a*(b\circ c)+a\circ b\circ c+a\circ c\circ b.
\end{gather*}
\end{example}

An alternative formula for symmetrizers can be proved
if $A$ is commutative
(\cf \cite{hoffman_multiple,hoffman_quasib}).
For $\pi=\set{B_1,\dots,B_k}\in\Pi_n$ and $a\in A^n$, let
(see \cref{aB} for the definition of $a_B$)
\eq{P_\pi(a)=a_{B_1}*\dots*a_{B_k}\in \SQS(A).}

\begin{theorem}
\label{sym formula}
Let $A$ be a commutative algebra and let $a\in A^n$. Then
\[
\SS(a_1,\dots,a_n)
=\sum_{\pi\in\Pi_n}
\mu(\pi)P_\pi(a),
\]
where
\[\mu(\pi)=(-1)^{n-\n{\pi}}\prod_{B\in\pi}(\n{B}-1)!,\qquad \pi\in\Pi_n.\]
\end{theorem}
\begin{proof}
A permutation $\si\in S_n$ is equivalently a set partition $\pi\in\Pi_n$ together with a cyclic ordering on each block of $\pi$.
We have $c(\si)=\n\pi$, hence $\sgn(\si)=(-1)^{n-\n\pi}$.
For a block $B\sbs[n]$, there are $(\n B-1)!$ cyclic orders.
Now we apply \cref{sym formula1}.
\end{proof}

Let us give an alternative proof of \cref{sym formula},
based on the M\"obius inversion of~\cref{qs-sh:eq}.

\begin{proof}[Second proof of \cref{sym formula}]
Set partitions of $[n]$ can be identified with equivalence relations on~$[n]$.
Define the partial order on $\Pi_n$
with $\pi\le\rho$ if $i\pi j$ implies $i\rho j$.
The least element of $\Pi_n$ is the discrete partition $\hat0=\set{\set1,\dots,\set n}$.
Expanding the quasi-shuffle product in $P_\pi(a)$, one may either keep two
blocks of $\pi$ separate or merge them using multiplication in~$A$.
Therefore
\eq{P_\pi(a)=\sum_{\rho\ge\pi}\SS_\rho(a).}
This formula is analogous to \cref{qs-sh:eq}.
We will apply the M\"obius inversion formula to this equation.
The M\"obius function $\mu:\Pi_n\xx\Pi_n\to\bZ$ of the poset $\Pi_n$
is defined by $\sum_{\pi:\si\le\pi\le\rho}\mu(\si,\pi)=\de_{\si,\rho}$.
Then
\[\sum_\pi\mu(\hat0,\pi)P_\pi(a)
=\sum_{\pi,\rho:\hat0\le\pi\le\rho}\mu(\hat0,\pi)\SS_\rho(a)
=\sum_\rho\de_{\hat0,\rho}\SS_\rho(a)=\SS_{\hat0}(a)=\SS(a_1,\dots,a_n).
\]
We have $\mu(\hat0,\pi)=\mu(\pi)$
(see \eg \cite[3.10.4]{stanley_enumerativea}).
\end{proof}

\subsection{Symmetric and quasisymmetric functions}
There is also a useful interpretation of symmetric quasi-shuffles in
terms of the Hopf algebras of symmetric and quasisymmetric functions.
Although this description is not needed for the quasi-modularity
results below, it gives a conceptual source for identities among
symmetrizers and allows standard identities for symmetric functions
to be transferred directly to quasi-shuffle algebras.

For $\al\in\bN^k$, let $\n\al=\sum_{i=1}^k\al_i$ and let
$x^\al=\prod_{i=1}^k x_i^{\al_i}$.
Define the graded algebra of bounded degree power series
\[\bQ\pser{x_1,x_2,\dots}_\bdd
=\bop_{d\ge0}\prod_{\n\al=d}\bQ x^\al
\sbs \bQ\pser{x_1,x_2,\dots}.\]
The graded subalgebra $\QSym\sbs \bQ\pser{x_1,x_2,\dots}_\bdd$
of \idef{quasisymmetric functions}
consists of power series $f\in \bQ\pser{x_1,x_2,\dots}_\bdd$
such that for all $\al\in\pN^k$
and all $\bn\in\pN^k$ with $n_1<\dots<n_k$,
the coefficients of the monomials $x^\al$ and
$x_{n_1}^{\al_1}\dots x_{n_k}^{\al_k}$ in $f$ coincide.
This algebra has the basis
\eq{M_\al
=\sum_{0<n_1<\dots<n_k}x_{n_1}^{\al_1}\dots x_{n_k}^{\al_k},
\qquad \al\in\pN^k,
}
with $M_{()}=1$, where $\pN^0=\set{()}$.
The graded algebra
\[\La=\bQ\pser{x_1,x_2,\dots}_\bdd^{S_\infty}\]
of symmetric functions is contained in $\QSym$.
Indeed, $\La$ has the basis consisting of monomial symmetric functions $m_\la$ for partitions $\la$ (see \eg \cite{macdonald_symmetric}).
For $\la$ of length $k$, we have
\[m_\la=\sum_{\al^+=\la}M_\al,\]
where $\al^+$ is the partition obtained by rearranging the entries of $\al\in\pN^k$.

The algebra $\QSym$ is a connected Hopf algebra with the coproduct
(\cf \cite{malvenuto_duality})
\[\de(M_\al)=\sum_{i=0}^k M_{(\al_1,\dots,\al_i)}\ts
M_{(\al_{i+1},\dots, \al_k)},\qquad \al\in\pN^k.\]
For example, $p_n=m_{(n)}=\sum_{i\ge1}x_i^n\in\La$ satisfies $\de(p_n)=p_n\ts1+1\ts p_n$,
hence $\La\sbs\QSym$ is a Hopf subalgebra
(\cf \cite[Example I.5.25]{macdonald_symmetric}).

Let $A$ be a (non-unital) algebra and let
$a=(a_1,\dots,a_m)\in A^m$ be a collection of pairwise commuting elements.
Let $\ol\QSym\sbs\QSym$ be the augmentation ideal.
Then the evaluation map
\eq{\phi:\ol\QSym\to A,\qquad
f\mto f(a_1,\dots,a_m,0,\dots),
}
is an algebra morphism.
By \cref{adj}, it induces a Hopf algebra morphism
\eq{\Phi:\QSym\to \QS(A).}
Explicitly, for $\al\in\pN^k$,
\eq{
\Phi(M_\al)
=
\sum_{0<i_1<\dots<i_r=k}
\phi(M_{\al_{1},\dots,\al_{i_1}})
\ts
\phi(M_{\al_{i_1+1},\dots,\al_{i_2}})
\ts\dots\ts
\phi(M_{\al_{i_{r-1}+1},\dots,\al_{k}}).
}
In particular, for $p_n=m_{(n)}=M_{(n)}$,
we obtain
\eq{\Phi(p_n)=\phi(p_n)=\sum_{i=1}^ma_i^n\in A.}
For $e_n=m_{(1^n)}=M_{(1^n)}$, we obtain
\eq{\Phi(e_n)=\sum_{\ov{k_1+\dots+k_r=n}{r\ge1,k_i\ge1}}e_{k_1}(a_1,a_2,\dots)\ts\dots\ts e_{k_r}(a_1,a_2,\dots).}
Since the elements $p_n$ generate the algebra $\La$ and $\Phi(p_n)\in A$,
the map $\Phi$ restricts to a Hopf algebra morphism
\eq{\Phi:\La\to\SQS(A).}

For a single element $a\in A$,
we have $\phi(M_\al)=0$ whenever $\ell(\al)\ge2$.
Therefore
\begin{equation}\label{Phi-M}
\Phi(M_\al)=a^{\al_1}\ts\dots\ts a^{\al_k},
\qquad \al\in\pN^k.
\end{equation}

\begin{remark}
Consider the non-unital algebra
\[
U=\bigoplus_{i\ge1}\bQ z_i,
\qquad
z_i\circ z_j=z_{i+j},
\]
which is isomorphic to the algebra $a\bQ[a]\subset\bQ[a]$ via $z_i\mapsto a^i$.
The Hopf algebra morphism
\[
\Phi:\QSym\to\QS(U),
\qquad
M_\al\mapsto z_{\al_1}\ts\dots\ts z_{\al_k},
\]
is an isomorphism, since it sends a basis to a basis.
The inverse of $\Phi$ was constructed in \cite[Theorem 3.4]{hoffman_algebra}.
The restriction of $\Phi$ to $\La$ induces an isomorphism $\Phi:\La\to\SQS(U)$.
\end{remark}

Using \cref{Phi-M} we obtain
\eq{
\Phi(p_n)=\Phi(m_{(n)})=a^n\in A,\qquad
\Phi(e_n)=\Phi(m_{(1^n)})=a^{\ts n}\in A^{\ts n}.
}
We can translate relations in $\La$ into relations in $\QS(A)$.
For example, let $Z:\QS(A)\to C$ be an algebra morphism.
The Newton identity \cite[I.2.11$'$]{macdonald_symmetric}
\[ne_n=\sum\nolimits_{i=0}^{n-1}(-1)^{n-i-1}e_{i}p_{n-i}\]
implies
\begin{equation}
nZ(a^{\ts n})
=\sum\nolimits_{i=0}^{n-1}(-1)^{n-i-1}Z(a^{\ts i})Z(a^{n-i}).
\end{equation}
Similarly, the formula
$e_n=\sum_{\la\prt n}(-1)^{n-\ell(\la)}z_\la\i p_\la$
\cite[I.2.14$'$]{macdonald_symmetric} implies
\begin{equation}
Z(a^{\ts n})=\sum_{\la\prt n}(-1)^{n-\ell(\la)}z_\la\i
\prod_{i=1}^{\ell(\la)}Z(a^{\la_i}).
\end{equation}

\section{Quasi-modularity}
\label{sec:qm}

We now apply the algebraic results of the previous section to multiple
$q$-zeta values.
The strategy is to find subalgebras $A\sbs\bQ\pser q^+$ for which
$Z(A)$ consists of quasi-modular forms and then apply
\cref{Z of TS} to the symmetric quasi-shuffle algebra $\TS(A)$.
We first treat level one using the algebra $R_2=\bQ[Q_2]$,
where $Q_2=\frac{q}{(1-q)^2}$, and then
pass to finite level by considering the $\la$-ring closure of~$R_2$.

\subsection{Quasi-modular forms}
We first fix our conventions for modular and quasi-modular forms,
including forms of finite level and their rational structures.
Let $\Ga\sbs\SL_2(\bZ)$
be a subgroup of finite index.
A \idef{modular form} of weight $k\ge0$ for $\Ga$ is a holomorphic function
$f:\bH\to\bC$ such that
\[
f\rbr{\frac{a\ta+b}{c\ta+d}}
=(c\ta+d)^k f(\ta)
\]
for every $\ga=\smat{a&b\\c&d}\in\Ga$,
and such that $f$ is holomorphic at every cusp of $\Ga$.

A \idef{quasi-modular form} of weight $k\ge0$
for $\Ga$ is a holomorphic function $f:\bH\to\bC$ for which there exist
holomorphic functions $f_0,\dots,f_p:\bH\to\bC$, with $f_0=f$, such that
\[
f\left(\frac{a\ta+b}{c\ta+d}\right)
=(c\ta+d)^k
\sum_{i=0}^p f_i(\ta)\rbr{\frac{c}{c\ta+d}}^i
\]
for every $\ga=\smat{a&b\\c&d}\in\Ga$,
and such that $f_0,\dots,f_p$ are holomorphic at every cusp of $\Ga$.
Let
\[\M(\Ga)=\bop_{k\ge0}\M_k(\Ga),\qquad
\QM(\Ga)=\bop_{k\ge0}\QM_k(\Ga)\]
be the graded algebras of modular and quasi-modular forms, respectively.
Then (see \cite{kaneko_quasimodular})
\[
\QM(\Ga)=\M(\Ga)[E_2].
\]
In particular, for $N\ge1$, consider the group
\[
\Ga_1(N)
=\sets{\smat{a&b\\c&d}\in\SL_2(\bZ)}
{c\equiv0\pmod N,\quad a,d\equiv1\pmod N},
\]
and let
\[
\M(N)=\M(\Ga_1(N)),
\qquad
\QM(N)=\QM(\Ga_1(N)).
\]
We have $\Ga_1(1)=\SL_2(\bZ)$ and
\[
\M(1)=\bC[E_4,E_6],
\qquad
\QM(1)=\bC[E_2,E_4,E_6].
\]
For $N\mid N'$, we have
\[
\M(N)\sbs\M(N'),
\qquad
\QM(N)\sbs\QM(N'),
\]
and we define
\[
\M(\infty)=\bigcup_{N\ge1}\M(N),
\qquad
\QM(\infty)=\bigcup_{N\ge1}\QM(N).
\]

Using $q=e^{2\pi i\ta}$, we have $\M(N),\QM(N)\sbs\bC\pser q$.
Let
\[\M(N;\bQ)=\M(N)\cap\QQ,\qquad
\QM(N;\bQ)=\QM(N)\cap\QQ.
\]
Then $\M(N)\iso\M(N;\bQ)\ts_\bQ\bC$
(see \eg \cite[Theorem 12.3.2]{diamond_modular}).
Since $\QM(N)=\M(N)[E_2]$ and $E_2\in\QQ$, it follows that
\[
\QM(N)\iso\QM(N;\bQ)\ts_\bQ\bC.
\]
The elements of
\[\QM=\QM(1;\bQ)=\bQ[E_2,E_4,E_6]
\]
will be called \idef{quasi-modular forms},
and the elements of
\[\QM(\infty;\bQ)=\QM(\infty)\cap\QQ\]
will be called \idef{quasi-modular forms of finite level}.

\subsection{Model algebras}
We next identify the algebra $R_2\sbs\bQ(q)$ that will produce
the level-one quasi-modular forms.
For a Laurent series $f=\sum_{i\in\bZ}f_iq^i\in\bQ\lser q$, let
\eq{\ord f=\inf\sets{i\in\bZ}{f_i\ne0}\in\bZ\sqcup\set\infty.}
In particular, $\ord(0)=\infty$.
For $f\in\bQ(q)$ and $c\in\bQ$, let $\ord_{c}f=\ord f(q+c)$.
Let
\eq{
\ord_\infty(f):=\ord f(q\i)=-\deg f,}
where $\deg(u/v)=\deg(u)-\deg(v)$ for $u,v\in\bQ[q]$.
Consider the involution
\[\ta:\bQ(q)\to\bQ(q),\qquad f\mto f(q\i).\]
Given an algebra $A\sbs \QQ$, let $A^+=A\cap \QQp$.

\begin{lemma}\label{deg<0}
Every element in $\bQ(q)^\ta\cap \QQp$ has degree $<0$.
\end{lemma}
\begin{proof}
Let $f/g\in \bQ(q)^\ta\cap \QQp$, where $f,g\in\bQ[q]$ are coprime.
Let $f=q^k h$, where $h\in\bQ[q]$ satisfies $h(0)\ne0$,
and let $m=\deg h$, $n=\deg g$.
Then
\[f/g=\frac{f(q\i)}{g(q\i)}
=q^{n-k-m}\frac{q^m h(q\i)}{q^n g(q\i)}\]
vanishes at $q=0$ only if $n-k-m>0$.
Then $\deg(f/g)=k+m-n<0$.
\end{proof}

Recall that $Q_k=D^{k-1}\frac q{1-q}$ for $k\ge1$.
In particular,
\[Q_1=\frac q{1-q},\qquad a:=Q_2=\frac q{(1-q)^2}=Q_1+Q_1^2.
\]
We also define $Q_0=1$.
Since $a\i=q+q\i-2$, we have
\eq{\bQ(q)^\ta
=\bQ(q+q\i)=\bQ(a),\qquad
\bQ[q^{\pm1}]^\ta=\bQ[q+q\i]=\bQ[a\i].
}
Consider the algebras
\eq{
R_2=\bQ[Q_2]\sbs R_1=\bQ[Q_1]\sbs R=\bQ\sbr{q,\tfrac 1{q-1}}\sbs\QQ.
}

\begin{lemma}
We have
\begin{enumerate}
\item
Every sequence $(f_i)_{i\ge0}$ in $R_1$
such that $\ord_{1}(f_i)=-i$ forms a basis of $R_1$.
\item
Every sequence $(f_i)_{i\ge1}$ in $R_1^+$
such that $\ord_{1}(f_i)=-i$ forms a basis of $R_1^+$.
\item
$R_1=\sets{f\in R}{\deg f\le 0}=\bop_{k\ge0}\bQ Q_k$.

\end{enumerate}
\end{lemma}
\begin{proof}
We have $R_1\sbs A:=\sets{f\in R}{\deg f\le 0}$.
The subspace
\[W_d=\sets{f\in A}{\ord_{1}f\ge -d}
=\sets{g/(q-1)^d}{g\in\bQ[q],\,\deg g\le d}\]
has dimension $d+1$.
The functions $f_0,\dots,f_d$ are contained in $W_d$
and are linearly independent.
Indeed, if $\sum_{i=0}^k a_if_i=0$ with $a_k\ne0$,
then $g_k=(q-1)^kf_k\in\bQ[q]$ satisfies $a_kg_k(1)=0$.
Therefore $\ord_1 f_k>-k$, a contradiction.
This implies that $f_0,\dots,f_d$ is a basis of $W_d$ for every~$d$,
hence $(f_i)_{i\ge0}$ is a basis of $A$.
In particular, $Q_1^i\in R_1$ satisfies $\ord_{1}Q_1^i=-i$,
hence $(Q_1^i)_{i\ge0}$ is a basis of~$A$.
This implies $R_1=A$ and (1).
The proof of (2) is similar.

Let us prove that $R_1=\bop_{k\ge0}\bQ Q_k$.
Define the Eulerian polynomials $P_k\in\bN[q]$ by
(see \eg \cite{hirzebruch_eulerian})
\[\frac{q P_k(q)}{(1-q)^{k+1}}
=\sum_{n\ge0}(n+1)^k q^{n+1}
=D^k\frac{q}{1-q}=Q_{k+1}.
\]
They satisfy $P_k(1)\ne0$, hence $\ord_{1}Q_k=-k$.
Therefore $(Q_{k})_{k\ge0}$ is a basis of~$R_1$.
\end{proof}

\begin{lemma}
We have
$R_2=R_1^\ta=R^\ta$.
\end{lemma}
\begin{proof}
The element $a=Q_2$ satisfies $\ta a=a$ and $a=Q_1+Q_1^2$,
hence $R_2\sbs R_1^\ta\sbs R^\ta$.
If $f\in R^\ta$, then $f/a^n\in \bQ[q^{\pm1}]^\ta=\bQ[a\i]$ for some $n>0$.
Therefore $f\in\bQ[a^{\pm1}]$.
Since $\deg f\le0$ by \cref{deg<0} and $\deg a^i=-i$, we conclude that $f\in\bQ[a]=R_2$.
Therefore $R_2=R_1^\ta=R^\ta$.
\end{proof}

\begin{lemma}
\label{R2-Q2k}
We have
\begin{enumerate}
\item
Every sequence $(f_i)_{i\ge0}$ in $R_2$
such that $\ord_{1}(f_i)=-2i$
forms a basis of $R_2$.
\item
Every sequence $(f_i)_{i\ge1}$ in $R_2^+$
such that $\ord_{1}(f_i)=-2i$
forms a basis of $R_2^+$.
\item
$R_2=\bop_{k\ge0}\bQ Q_{2k}$.
\end{enumerate}
\end{lemma}
\begin{proof}
Let $P=\sum_{i=0}^d a_ix^i\in\bQ[x]$ have degree $d$.
Since $\ord_{1}(Q_2^i)=-2i$,
we obtain
\[\ord_{1}P(Q_2)=-2d=-2\deg P.\]
Writing $f_i=P_i(Q_2)$ for $P_i\in\bQ[x]$,
the assumption $\ord_1(f_i)=-2i$ implies $\deg P_i=i$.
Hence $(P_i)_{i\ge0}$ is a basis of $\bQ[x]$, and therefore
$(f_i)_{i\ge0}$ is a basis of $R_2=\bQ[Q_2]$.
The proof of (2) is similar.

Let us prove that $R_2=\bop_{k\ge0}\bQ Q_{2k}$.
If $f\in\bQ(q)$, then $D\ta f=q\frac d{dq}f(q\i)=-q\i f'(q\i)=-\ta Df$.
Hence $D^2\ta=\ta D^2$.
Since $\ta Q_2=Q_2$, we conclude that $\ta Q_{2k}=Q_{2k}$, hence $Q_{2k}\in R_1^\ta=R_2$.
We have seen that $\ord_{1} Q_k=-k$ for all $k\ge0$.
Therefore $(Q_{2k})_{k\ge0}$ is a basis of $R_2$.
\end{proof}

\subsection{\tpdf{$\la$}{Lambda}-ring closures}
To pass from level-one to finite-level quasi-modular forms,
we enlarge the preceding algebras using the Adams operations $\psi^n(f)=f(q^n)$.
We determine the $\tau$-invariant parts of the corresponding closures explicitly;
this will provide the model algebra for \cref{thm2}.

Given a $\la$-ring $B$ and a subring $R\sbs B$,
define the $\la$-ring closure $\bbR\sbs B$,
to be the intersection of all $\la$-subrings of $B$ that contain $R$.
In particular, consider the $\la$-ring $\QQ$ with Adams operations $\psi^n(f)=f(q^n)$,
and the subrings
\[
R_2=\bQ[Q_2]\sbs R_1=\bQ[Q_1]\sbs R=\bQ\sbr{q,\tfrac 1{q-1}}\sbs \QQ.
\]
Then
\[
\bbR_2=\bQ\sbr{\psi^n(Q_2)\col n\ge1},\qquad
\bbR_1=\bQ\sbr{\psi^n(Q_1)\col n\ge1},\qquad
\bbR=\bQ\sbr{q,\tfrac 1{q^n-1}\col n\ge1}
.\]

\begin{theorem}\label{R2 closure}
We have (for $a=Q_2=\frac q{(1-q)^2}$)
\[\bbR_2=(\bbR_1)^\ta=(\bbR)^\ta
=\bQ[a,\psi^n(a)/a^n\col n\ge2].
\]
\end{theorem}

\begin{proof}
Let $A=\bQ[a,\psi^n(a)/a^n\col n\ge2]$.
We will show that
\[\bbR_2\sbs(\bbR_1)^\ta\sbs(\bbR)^\ta\sbs A\sbs\bbR_2,\]
which will imply that all these algebras are equal.
The rings $(\bbR_1)^\ta$ and $(\bbR)^\ta$ are $\la$-rings,
since the Adams operations commute with $\ta$.
We have $R_2\sbs (\bbR_1)^\ta$, hence $\bbR_2\sbs (\bbR_1)^\ta\sbs (\bbR)^\ta$.

Let $f\in(\bbR)^\ta$.
Then $f\cdot (q^n-1)^k\in\bQ[q]$ for some $k,n\ge1$.
Therefore $f/\psi^n(a)^k\in\bQ[q^{\pm 1}]^\ta=\bQ[a\i]$,
hence $f=\psi^n(a)^k g(a\i)$ for some polynomial $g=\sum_{i=0}^d g_i x^i\in\bQ[x]$ of degree $d\ge0$.
We have $\deg(f)\le 0$ and
$\deg \rbr{\psi^n(a)^k g(a\i)}=-kn+d$, hence $d\le kn$.
Therefore
\[f=\rbr{\frac{\psi^n(a)}{a^n}}^k\cdot a^{kn}g(a\i)\in A.\]

To prove that $A\sbs\bbR_2$ we need to show that
$\psi^n(a)/a^n\in\bbR_2$.
We have $\psi^n(a\i)=q^n+q^{-n}-2\in \bQ[q^{\pm1}]^\ta=\bQ[a\i]$.
Therefore $\psi^n(a\i)$ is a polynomial in $a\i$ of degree $n$, with the leading coefficient $1$.
This implies that $f_n(a)=a^n \psi^n(a\i)$
is a polynomial in $a$ with $f_n(0)=1$.
Since $\gcd(a^n,f_n)=1$,
we have $a^n u+f_nv=1$ for some polynomials $u,v$ in $a$,
hence
$1/f_n=\frac {a^n}{f_n} u+v=\psi^n(a)u+v\in\bbR_2$.
Therefore $\psi^n(a)/a^n=1/f_n\in\bbR_2$.
\end{proof}

\subsection{Quasi-modularity results}
We now combine the algebraic criterion of \cref{Z of TS} with the
description of the model algebras above.
For $R_2$ the required one-variable $Z$-values are Eisenstein series,
which gives the level-one result directly.
For the $\lambda$-ring closure, the corresponding statement follows
from a quasi-modularity theorem for cyclotomic rational
functions.

\begin{customthm}{Theorem \ref{thm1}}
The image of the map $Z:\TS(R_2^+)\to\QQ$
is equal to $\QM(1;\bQ)$.
\end{customthm}
\begin{proof}
We know that $E_{2k}\in\QM(1;\bQ)=\bQ[E_2,E_4,E_6]$ for all $k\ge1$.
Therefore $Z(Q_{2k})\in\QM(1;\bQ)$ for all $k\ge1$.
Since $R_2^+=\bop_{k\ge1}\bQ Q_{2k}$ by \cref{R2-Q2k},
we conclude that $Z(R_2^+)\sbs\QM(1;\bQ)$.
Therefore $C=Z(\TS(R_2^+))\sbs\QM(1;\bQ)$ by \cref{Z of TS}.
Since $\TS(R_2^+)$ is a subalgebra of $\QS(R_2^+)$
and $E_2,E_4,E_6\in C$, we conclude that $C=\QM(1;\bQ)$.
\end{proof}

To prove \cref{thm2}, we will require
a reformulation of \cite[Theorem 1.1]{shintani_generalization}.
Let $\Phi_d\in\bZ[q]$ be the cyclotomic polynomial defined by
$q^n-1=\prod_{d\mid n}\Phi_d(q)$,
and let $\phi(d)=\deg\Phi_d$ be Euler's totient function.
We have
\[\Phi_d(q\i)
=(-1)^{\de_{d,1}}q^{-\phi(d)}\Phi_d(q).\]

\begin{theorem}
Let $k,n\ge1$ and let $Q\in \bQ[q]^+$
satisfy $\deg Q<k\phi(n)$ and $Q/\Phi_n^k\in\bQ(q)^\ta$.
Then $Z(Q/\Phi_n^k)\in\QM(n;\bQ)$.
\end{theorem}

Note that the condition $\deg Q<k\phi(n)$ is automatic by \cref{deg<0}.

\begin{customthm}{Theorem \ref{thm2}}
The image of the map
$Z:\TS(\bbR_2^+)\to\QQ$ is contained in $\QM(\infty;\bQ)$.
\end{customthm}
\begin{proof}
Every $f\in\bbR=\bQ\sbr{q,\tfrac 1{q^n-1}\col n\ge1}$
admits a unique partial fraction decomposition
\[f=P+\sum_{n,k\ge1}\frac{Q_{n,k}}{\Phi_n^k},\qquad
\deg Q_{n,k}<\phi(n),\]
where $P,Q_{n,k}\in\bQ[q]$.
If $f\in\bbR_2^+=(\bbR^\ta)^+$, then $\deg f<0$ by \cref{deg<0}.
Since every proper partial-fraction term has negative
degree, this implies $P=0$.
Therefore
\[f=\tfrac12(f+\ta f-f(0))
=\tfrac12\sum_{n,k\ge1}
\rbr{
\frac{Q_{n,k}(q)+(-1)^{k\de_{n,1}}q^{k\phi(n)}Q_{n,k}(q\i)}
{\Phi_n(q)^k}
-\frac{Q_{n,k}(0)}{\Phi_n(0)^k}}
.\]
Every summand
\[
g=\frac{Q(q)}{\Phi_n(q)^k}
:={
\frac{Q_{n,k}(q)+(-1)^{k\de_{n,1}}q^{k\phi(n)}Q_{n,k}(q\i)}
{\Phi_n(q)^k}
-\frac{Q_{n,k}(0)}{\Phi_n(0)^k}}
\]
is contained in $\bQ(q)^\ta$
and satisfies $g(0)=0$.
Therefore $\deg(g)<0$ by \cref{deg<0}, hence $\deg Q<k\phi(n)$.
Applying the above theorem, we conclude that $Z(g)\in\QM(\infty;\bQ)$.
Therefore $Z(f)\in \QM(\infty;\bQ)$,
hence $Z(\bbR_2^+)\sbs\QM(\infty;\bQ)$.
Now we apply \cref{Z of TS}.
\end{proof}


\bibliographystyle{halpha}
\bibliography{biblio}
\end{document}